 \documentclass[12pt]{amsart}

\usepackage{hyperref}
\hypersetup{
	bookmarks=true,         
	unicode=false,          
	pdftoolbar=true,        
	pdfmenubar=true,        
	pdffitwindow=false,     
	pdftitle={},    
	pdfauthor={},     
	pdfsubject={},   
	pdfcreator={Me},   
	pdfproducer={},  
	pdfkeywords={}, 
	pdfnewwindow=true,      
	colorlinks=false,       
	linkcolor=red,          
	citecolor=green,        
	filecolor=magenta,      
	urlcolor=cyan           
}
\usepackage[pdftex]{graphicx} 
\usepackage{color}
\usepackage[T1]{fontenc}
\usepackage{lmodern}
\usepackage[english]{babel} 
\usepackage[utf8]{inputenc}

\usepackage{upgreek}
\usepackage{mathabx}

\usepackage{amsfonts}
\usepackage{verbatim}
\usepackage{amsmath}
\usepackage[all, cmtip]{xy}\usepackage{amssymb}
\usepackage{amsthm}
\usepackage{extpfeil}
\usepackage{bbm}
\usepackage{paralist}
\usepackage[T1]{fontenc}

\usepackage{amscd}

\usepackage{nicefrac}
\usepackage{microtype}
\usepackage{epstopdf}

\usepackage{mathrsfs}
\usepackage[font=small,labelfont=bf]{caption}
\usepackage[labelformat=simple]{subcaption}

\makeatletter
\newtheorem*{rep@theorem}{\rep@title}
\newcommand{\newreptheorem}[2]{%
\newenvironment{rep#1}[1]{%
 \def\rep@title{#2~\ref{##1}}%
 \begin{rep@theorem}}%
 {\end{rep@theorem}}}
 
\makeatother

\theoremstyle{plain}
\newtheorem*{thm*}{Theorem}
\newtheorem{thm}{Theorem}[section]

\newreptheorem{mthm}{Theorem}
\newreptheorem{mcor}{Corollary}
\newreptheorem{mlem}{Lemma}
\newreptheorem{thm}{Theorem}

\newtheorem{cor}[thm]{Corollary}
\newtheorem{lem}[thm]{Lemma}
\newtheorem*{lem*}{Lemma}
\newtheorem{prp}[thm]{Proposition}

\theoremstyle{definition}
\newtheorem{dfn}[thm]{Definition}

 \newcommand{\R}{\mathbb{R}}

\newcommand{\vanish}[1]{}

\def\({\left(}
\def\){\right)}
\def\no={\,{\,|\!\!\!\!\!=\,\,}}

\def\no={\,{\,|\!\!\!\!\!=\,\,}}

\def\conv{\text{\rm{conv}}}

\newcommand{\xqedhere}[2]{%
  \rlap{\hbox to#1{\hfil\llap{\ensuremath{#2}}}}}

\newcommand\note[1]{\textbf{\color{red}}}
\newcommand\Defn[1]{\textbf{\emph{#1}}}
\newcommand{\cm}[1]{}

\newcommand\mr[1]{\mathrm{#1}}

\renewcommand\emptyset{\varnothing}

\DeclareMathOperator{\lk}{Lk}
\DeclareMathOperator{\st}{St}

\DeclareMathOperator{\Cone}{Cone}
\DeclareMathOperator{\Star}{St}

\title{Normal crossing immersions, cobordisms and flips}
\author{Karim Adiprasito}
\author{Gaku Liu}

\address{Sorbonne Université and Université Paris Cité, CNRS, IMJ-PRG, F-75005 Paris, France}
\email{karim.adiprasito@imj-prg.fr}

\address{Department of Mathematics, University of Washington, Seattle, USA}
\email{gakuliu@uw.edu}

\thanks{K.A.~is supported by the
	Centre National de Recherche Scientifique and  Horizon Europe ERC Grant number: 101045750 / Project acronym: HodgeGeoComb, G.L. was supported by NSF grant 1440140.}
\date{\today}

\begin{document}
\begin{abstract}
We study various analogues of theorems from PL topology for cubical complexes. In particular, we characterize when two PL homeomorphic cubulations are equivalent by Pachner moves by showing the question to be equivalent to the existence of cobordisms between generic immersions of hypersurfaces. This solves a question and conjecture of Habegger and Funar. 
\end{abstract}

\maketitle

\section{Introduction}
Pachner's influential theorem proves that every PL homeomorphim of triangulated manifolds can be written as a combination of local moves, the so-called Pachner moves (or bistellar moves). Geometrically, they correspond to exchanging, inside the given manifold, one part of the boundary of a simplex by the complementary part. For instance, in a two-dimensional manifold, a triangle can be replaced by three triangles with a common vertex, and two triangles sharing an edge can be replaced by two different triangles sharing the edge connecting the vertices opposite the original edge. 

Cubical complexes, popularized for their connection to low-dimensional topology and geometric group theory (see for instance \cite{GromovHG}), have analogous moves called \emph{cubical Pachner moves}. However, the situation is not as simple as the simplicial case. Indeed, a cubical Pachner move can never change, for instance, the parity of the number of facets. But cubical polytopes, for instance in dimension 3, can have odd numbers of facets as well as even; see \cite{SZ} for a few particularly notorious constructions.

Hence, there are at least two classes of cubical 2-spheres that can never be connected by cubical Pachner moves. Indeed, Funar \cite{Funar} provided a conjecture for a complete characterization for the cubical case that revealed the problem's depth
depth: Every cubical manifold of dimension $d$ has associated an immersed normal hypersurface. To construct this, consider an edge of the cubical complex; transversal to it, draw a $(d-1)$-dimensional disk. Continue drawing through adjacent edges, that is, edges that are in a common square, but have no vertex in common. Repeating this for every edge yields the desired hypersurface. We refer to \cite{Funar} for more on these notions.

Note that the union of these normal hypersurfaces form precisely the codimension one-skeleton of the dual cell decomposition. Notice further that every self-intersection of that hypersurface is normal crossing; the intersections are transversal.

We call two normal crossing hypersurfaces \emph{cobordant} if they are cobordant as normal crossing hypersurfaces: If $H$, $H'$ are normal crossing hypersurfaces in a manifold $M$, then they are cobordant if there is a normal crossing hypersurface in $M\times [0,1]$ that restricts to $H$ and $H'$ in the two boundary components.

Funar conjectured, and we prove:

\begin{thm}\label{thm:main}
For two PL cubulations $X_0$, $X_1$ of the same manifold $M$, the following three conditions are equivalent:\
\begin{compactenum}[(1)]
\item The normal surfaces of $X_0$, $X_1$ of a manifold $M$ are cobordant as generic immersions.
\item There is a PL cubulation of $M\times [0,1]$ such that $M\times \{0\} \cong X_0$, $M\times \{1\} \cong X_1$.
\item $X_0$ and $X_1$ are related by cubical Pachner moves.
\end{compactenum}
\end{thm}

This resolves a problem of Habegger and Funar \cite{Funar}, see the remark on top of his third page. Habegger in fact assumed that only a PL homeomorphism is necessary for two cubulations to be related through cubical Pachner moves; it was Funar that recognized the importance of normal crossing cobordisms for the problem, and in particular also computed these cobordism groups for spheres \cite{FG}. In fact, Funar proved almost all implications, except one: that (2) implies (3). For surfaces and $1$-dimensional manifolds, the problems were solved by Funar \cite{FunarS} and independently sketched by Thurston \cite{Thurston}.

\section{Preliminaries}

\subsubsection{Basic definitions}

A \emph{polyhedral complex} is a complex of polytopes where the intersection of any two polytopes is a union of faces of each polytope, and where we do not allow identification of faces of the same polytope. A \emph{simplicial complex} and a \emph{cubical complex} are complexes whose polytopes are all simplices and cubes, respectively. These are also called \emph{triangulations} and \emph{cubulations}, respectively. Given a polyhedral complex $A$ and a subcomplex $B$ of $A$, we let $A \setminus B$ be the complex generated by the facets of $A$ that are not in $B$. Given a complex $A$ and a face $\tau$ of $A$, we let $A - \tau$ be the complex consisting of faces of $A$ that do not contain $\tau$.

Given a polyhedral complex $A$ and a face $\tau$ of $A$, the \emph{star} $\Star_\tau A$ of $\tau$ in $A$ is the subcomplex generated by all faces of $A$ containing $\tau$. A disk is a PL disk if it admits a PL homeomorphism to the simplex. A polyhedral manifold is a polyhedral complex homeomorphic to a manifold, and a PL manifold is a polyhedral manifold in which the star of any vertex is a PL ball. Hence a PL triangulation resp.\ PL cubulation is a triangulation resp.\ cubulation in which every star is a PL ball.

The \emph{link} $\lk_\tau A$ of a face $\tau$ in $A$ is, in terms of the poset, the principal filter of $\tau$ in $A$. Geometrically, it can be understood as follows: If $\tau$ is a vertex, then its link is the intersection of $A$ with a sufficiently small metric sphere with center $\tau$. Here, a metric sphere is the boundary of a metric ball.
Provided this metric sphere is contained entirely in the star of $\tau$, this yields a spherical polyhedral complex whose face poset is the filter discussed above. 

The link of a general face $\tau$ is defined iteratively, by considering one of its vertices $v$, and defining $\lk_\tau A$ to be the link of $\lk_v \tau$ in $\lk_v A$. The link of the empty set in a complex is naturally the complex itself.

\subsubsection{Cubical Pachner moves} First, it is useful to have different moves available. Let us first recall the definitions of Pachner moves.

A simplicial Pachner move in a $d$-dimensional simplicial complex $S$ picks a $d$-dimensional subcomplex of $S$ isomorphic to a subcomplex of the boundary of the $(d+1)$-simplex, and replaces it with the complementary subcomplex of that simplex. Here is the cubical version.

\begin{dfn}[Cubical Pachner moves] 
If $X$ is a cubical manifold, that is, a cubulation of a manifold, and $C$ is a cube of one dimension higher, then a cubical Pachner move consists of removing a pure, full dimensional, contractible subcomplex of $X$ isomorphic to a subcomplex of $\partial C$ and replacing it with the complement.
\end{dfn}

A good way to think about Pachner moves is in terms of cobordisms and shellings. A \emph{relative complex} $P = (A,B)$ is a pair of polyhedral complexes where $B \subseteq A$. The \emph{faces} of $P$ are the faces of $A$ that are not faces of $B$. Geometrically, we identify $P$ with $\lvert\lvert A \rvert\rvert$. 

Consider a relative polyhedral complex $P=(A,B)$ and a facet $F$ of $P$. We say that $P$ \emph{shells} to $P\setminus F := ((A \setminus F) \cup B, B)$ if $F\cap ((A \setminus F) \cup B)$ is shellable of dimension $\dim F -1$ and $P\setminus F$ is of the same dimension as $P$ or has no faces. 
 
A relative complex $(A,B)$ is shellable if iterated shelling steps reduce it to $(B,B)$. A complex $A$ is shellable if $(A,\emptyset)$ is shellable.
 
With this, we have:

\begin{lem}\label{lem:cobtoshell} If a triangulation/cubulation of $(M\times [0,1],M\times\{0\})$ is shellable then $M \times \{0\}$ and $M \times \{1\}$ are related through Pachner moves. 

Conversely, if two cubulations or triangulations are related by (cubical) Pachner moves, then they related by a shellable cobordism. (That is, there is a triangulation/cubulation of $M \times [0,1]$ where $M \times \{0\}$ and $M \times \{1\}$ are the given triangulations/cubulation and $(M\times [0,1],M\times\{0\})$ is shellable.)
\end{lem}

\begin{proof}
Every shelling step exactly corresponds to a Pachner move: Consider the change in $M \times \{1\}$ in such a step. Observe then that the part of the removed polytope (cube or simplex) in the shelling step that intersects $M \times \{1\}$ is replaced by the remaining facets of the boundary, as desired.
\end{proof}

Concerning regular subdivisions of convex disks, that is, regions of linearity of a convex function, the following is useful to keep in mind:

\begin{prp}[\cite{BruggesserMani}]\label{prp:BM}
A regular subdivision of a convex disk $D$ in $\R^d$ is shellable. If, moreover, $v$ is any point outside of $D$, and $M$ is the part of the boundary of $D$ illuminated by $v$, that is, those facets that can be connected to $v$ with a line segment that does not intersect $D$ in the interior, then $(D,M)$ is shellable.
\end{prp}

\subsubsection{Cubical stellar and cubical derived subdivisions}

We next describe the cubical analogue of stellar subdivisions and derived subdivisions. Recall that a stellar subdivision of a polyhedral complex $S$ picks a face $\tau$ of $S$ and considers the complex $S-\tau$. It then glues $\Cone(\Star_S \tau - \tau)$ to $S-\tau$ along $\Star_S \tau - \tau$. In geometric situations the conepoint is often placed inside $\tau$. Here is the cubical variant.

\begin{dfn}[Cubical stellar subdivision]
Let $C$ be a polyhedral complex in $\R^d$, and let $\tau$ be any face of $C$. Let $x_\tau$ denote a point anywhere in the relative interior of $\tau$, and let $\lambda$ be any number in the interval $(0,1)$. Define
\begin{align*}
\text{c-st}(\tau,C)\ :=\ & (C-\tau) \cup \\
&\{\conv (\sigma\cup (\lambda\sigma+(1-\lambda)x_\tau)): \sigma \in \mr{St}_\tau C -\tau \} \cup \\
& (\lambda \mr{St}_\tau C + (1-\lambda) x_\tau).
\end{align*}
The complex $\text{c-st}(\tau,C)$ is the \Defn{cubical stellar subdivision}, or \Defn{c-stellar subdivision}, of $C$ at $\tau$.
\end{dfn}

\begin{figure}[hbt]
	\includegraphics[width=13.5cm]{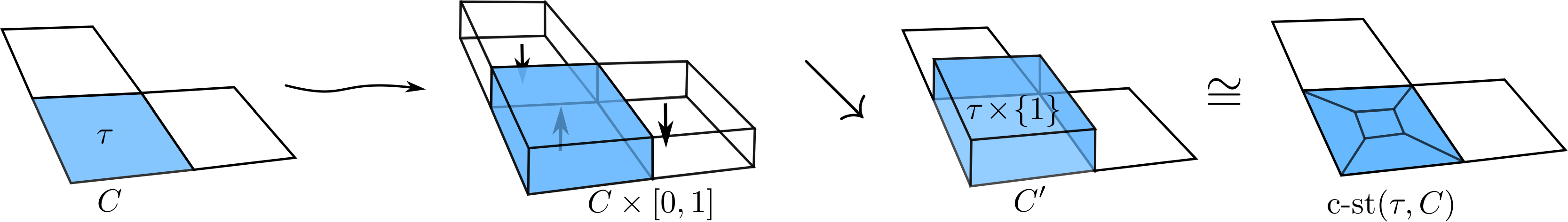}
	\caption{A cubical stellar subdivision of a complex $C$, seen as a subcomplex of the product of $C$ with an interval.}
	\label{f:over}
\end{figure}

A useful way to think about a cubical stellar subdivision is to think of it as a stellar subdivision, that is, the removal of $\tau$ from $C$ and coning over the boundary of the hole left over, followed by "\emph{cutting off}" the conepoint. In particular, if the link of $\tau$ is regular, then we can indeed achieve this visually by introducing new linear constraints at the apex vertex of the cone.

Finally, we look at cubical derived subdivisions. In the simplicial case, these derived subdivisions arise by performing stellar subdivisions at the faces in order of decreasing dimension. The analogue here is:

\begin{dfn}[Cubical derived subdivision]
Let $C$ denote any polytopal $d$-complex, with faces ordered by decreasing dimension.
A \Defn{cubical derived subdivision}, or \Defn{c-derived subdivision}, is any subdivision $C$ obtained by first c-stellar subdividing all $d$-faces of $C$, then c-stellar subdividing the resulting complex at all the original $(d-1)$-faces of $C$, then all original $(d-2)$-faces, and so on up to the faces of dimension~$1$.
\end{dfn}

Again, in the way they were introduced, it is natural to think of stellar subdivisions as cobordisms, obtained by attaching $\mr{St}_\tau C \times [0,1]$ to the complex in question. (Cubical) stellar subdivisions have several nice properties; for example, they preserve shellability.

\begin{lem}\label{lem:preserve}
Consider a shellable pair $(X,Y)$ that is relatively shellable. If $(X',Y')$ is a stellar subdivision of this pair, then it is shellable as well.
\end{lem}

\begin{proof}
	Restricted to any facet of $X$, the stellar subdivision is regular, and hence shellability follows from Proposition~\ref{prp:BM}.	
\end{proof}

 More importantly, it is useful to note the following corollary of the main theorem.

\begin{cor}
Every cubical stellar subdivision of a PL cubical manifold $X$ can be achieved by cubical Pachner moves.
\end{cor}

\begin{proof}
By attaching $\mr{St}_F X \times [0,1]$ to $X$ with the identification $\mr{St}_F X \times \{0\} \sim \mr{St}_F X$, we obtain a cubical cobordism of $X$ to the cubical stellar subdivision of $X$ at $F$. This gives the desired by the main theorem.
\end{proof}

Let us observe a weaker version, which we can prove immediately.

\begin{prp}\label{lem:stellar}
Every cubical stellar subdivision of a PL cubical manifold $X$ can be achieved by cubical Pachner moves provided, assuming the link of the subdivided face is shellable.
\end{prp}

\begin{proof}
$\mr{St}_F X \times [0,1]$ is shellable relative to $\mr{St}_F X \times \{0\}$ by assumption.
\end{proof}

Note that the converse does not hold: not every cubical Pachner move is obtainable by cubical stellar moves. Specifically, note that cubical stellar moves are topologically ``boring'' with respect to the associated normal hypersurfaces; they only introduce null-homotopic spheres as hypersurfaces, or remove them. But general cubical Pachner moves can change the topology or intersection patterns of the immersed surfaces. For example, a loop in $S^2$ with two self-intersection is normal crossing equivalent to the loop with no self intersection, but this cannot be obtained by stellar moves only. 
%

This is contrary to the case of simplicial complexes, where simplicial Pachner moves and stellar subdivisions and inverses are of the same strength, at least for PL manifolds: two PL triangulations of the same manifold are related by Pachner moves if and only if they are related by stellar moves. Beyond PL manifolds, things become more complicated, and stellar subdivisions and inverses can relate triangulations that are unrelated by Pachner moves.

\subsubsection{From cubical-stellar to stellar and back: the confinement map}\label{ssc:confinement} 

Let us compare cubical stellar and stellar subdivision for the briefest moment: Recall, the cubical stellar subdivision of a complex $C$ at a face $\tau$ is
\begin{align*}
	\text{c-st}(\tau,C)\ :=\ & (C-\tau) \cup \\
	&\{\conv (\sigma\cup (\lambda\sigma+(1-\lambda)x_\tau)): \sigma \in \mr{St}_\tau C -\tau \} \cup \\
	& (\lambda \mr{St}_\tau C + (1-\lambda) x_\tau).
\end{align*}

Compared to that, the stellar subdivision is
\begin{align*}
	\text{st}(\tau,C)\ :=\ & (C-\tau) \cup \\
	&\{\conv (\sigma\cup x_\tau)): \sigma \in \mr{St}_\tau C -\tau \}
\end{align*}

In other words, the stellar subdivision is the limit of the cubical stellar subdivision when $\lambda\rightarrow 0$.

In particular, using this contraction we obtain a polyhedral map, called \Defn{confinement map}, 
	\[\varphi_\tau:\text{c-st}(\tau,C)\ \longrightarrow\ \text{st}(\tau,C)\]
that is, a map that sends polyhedra to polyhedra and in particular induces a map of posets. The following is clear: Given a polyhedral complex $I$, let $I^{>0}$ denote the filter generated by all faces of positive dimension.

\begin{lem}\label{lem:confinement}
The confinement map is injective on $\text{st}(\tau,C)^{>0}$, that is, on the order filter and its preimage, the map is one to one as a map of posets. 

Moreover, the preimage of a vertex is a vertex, unless the vertex is $x_\tau$, in which case it is $\mr{St}_\tau C$.
\end{lem}

In other words, the source of non-injectiveness is the 0-skeleton of $\text{st}(\tau,C)$, that is, the vertices. This lemma allows us to associate to any iterated stellar subvision of $C$ a canonical cubical stellar subdivision. If, say, $C$ is stellarly subdivided at a face $\tau$, then there is an associated cubical stellar subdivision of $C$ carried out in $\tau$.

Now if we consider a further stellar subdivision of the resulting complex $C'$, then the preimage of the subdivided face under the confinement map is a face (by Lemma~\ref{lem:confinement}) that is we can analogously perform a cubical stellar subvision on. 

Hence, to an iterated simplicial stellar subdivision we have an associated cubical stellar subdivision, and an associated confiment map. Lemma~\ref{lem:confinement} still applies to this iterated confinement map: The map is injective on the order filter of faces of positive dimension.

\subsubsection{A non-useful subdivision}

There is another kind of subdivision in cubical manifolds. While simple, it is, from the view of cobordisms of normal crossing hypersurfaces, boring. But we record it here for curiosity's sake. Specifically, consider a cubulation $X$ of a $d$-manifold $M$. The normal crossing hypersurface of $X$ divides every $d$-cube of $M$ into $2^d$ smaller $d$-cubes. It divides every edge into two components, each of which can be cancelled out with each other in a cobordism. We obtain:

\begin{prp}
The hypersurface of the above subdivision is cobordant to the trivial hypersurface; the empty one.
\end{prp}

\section{Whitehead's neighborhood theorem}

Before we start, let us observe a basic proposition, going back to of Whitehead \cite{Whitehead}. Recall: In a polyhedral complex, an elementary collapse is the removal of a \Defn{free} face, that is, a face $A$ contained strictly only in one other face $B$. Both the free faces and the face $B$ are removed together, yielding a smaller cell complex with two faces less. Let us call the face $B$ \Defn{cofree}.

A collapse is the combination of elementary collapses. 

\begin{prp}[See also \cite{MST}]\label{prp:Whiteh}
Consider a collapse $X\searrow Y$ of polyhedral (resp.\ cubical) complexes of dimension~$d$ such that stars $\mr{St}_B X$ of cofree faces $B$ are of dimension $d$ and shellable relative to $Y \cap \mr{St}_B X$.

Then the second simplicial (resp.\ cubical) derived subdivision $X''$ of $X$ shells to the second derived subdivision of $Y''$.
\end{prp}

\begin{proof}
Consider first the case of simplicial derived subdivision: Define the \emph{neighborhood} $N_A X$ of a subcomplex $A$ in a polyhedral complex $X$ to be the closure of all faces of $X$ containing $A$.

Consider now the vertices of the first derived subdivision; as the original complex admits a collapse, the vertices come in pairs corresponding to those collapsing pairs.

Now, we consider the neighborhoods $(N_\gamma)_{\gamma=1,\cdots m}$ of these pairs in the second derived subdivsion $X''$ of $X$. 

The basic observation is that the neighborhood $N_\gamma$ of the barycenter of $A$ in $X''$ can be shelled relative to the complexes consisting of the remaining neighborhoods union $Y''$, that is, $N_\gamma$ is shellable relative to 
\[N_\gamma \cap \left(Y'' \cup \bigcup_{\gamma'>\gamma} N_{\gamma'} \right).\]
This is due to $\st_A X$, and therefore its second derived subdivision, being shellable. Therefore, the order of collapses gives a shelling of $X''$ to the induced subdivision of $Y$.

\emph{For the cubical case}, we note that the image under the confinement map has the desired shellability property, as it is simply the second simplicial derived subdivision. It is easy to see that the fibers of the confinement map at vertices individually are shellable, as they are cubical derived subdivisions of shellable stars, and derived subdivisions preserve shellability  by Lemma~\ref{lem:preserve}. Hence the second cubical derived subdivision has the desired shelling property as well.
\end{proof}

We close this section with another simple fact:

\begin{lem}\label{lem:starcollapse}
A star $\mr{St}_\tau C$ collapses to any of its vertices.
\end{lem}

We need an auxiliary lemma. 

\begin{lem}\label{lem:starcollapse2}
Consider a polytope $P$, and $F$ and face of its boundary. Then $P$ collapses to $\mr{St}_F \partial P$.
\end{lem}

\begin{proof}
Proposition~\ref{prp:BM} gives this in the context of shellings, but it is an easy and classical fact that a shelling induces a collapse.
\end{proof}

\begin{proof}[\textbf{Proof of Lemma~\ref{lem:starcollapse}}]
Consider the minimal face $T$ of $\mr{St}_\tau C$ containing $\tau$ and the vertex $v$ we wish to collapse to.

Consider now a maximal face $F$ of $\mr{St}_\tau C$, and assume $F$ is not $T$. Following Lemma~\ref{lem:starcollapse2}, $F$ collapses to $\mr{St}_\tau \partial F$ and hence
$\mr{St}_\tau C$ collapses down to $(\mr{St}_\tau C \setminus F) \cup \mr{St}_\tau \partial F $, the complex obtained by removing $F$ and all its faces from $\mr{St}_\tau C$, and adding $\mr{St}_\tau \partial F$.

Since 
\[\mr{St}_\tau\left( (\mr{St}_\tau C \setminus F) \cup \mr{St}_\tau \partial F\right)\ =\ (\mr{St}_\tau C \setminus F) \cup \mr{St}_\tau \partial F,\]
we can iterate this process until we collapsed $\mr{St}_\tau C$ to $T$. Now, we use Lemma~\ref{lem:starcollapse2} again and collapse the latter to $v$.
\end{proof}

\section{Cubical Pachner Theorem}

We now prove the main theorem. The directions $(1) \implies (2)$ and $(3) \implies (1)$ are shown in \cite{Funar}; it remains to prove $(2) \implies (3)$. Let $M$ be a manifold with two PL cubulations $X_0$ and $X_1$, and let $Y$ be a PL cubulation of $M \times [0,1]$ such that $Y \cap (M \times \{0\}) \cong X_0$, $Y \cap (M \times \{1\}) \cong X_1$.

We have the following lemma, following from the main result of \cite{AB-SSZ}.

\begin{lem}
Let $Y$ be a PL cubical manifold with boundary. After some iterated cubical stellar subdivison, the stars of faces in $Y$ are shellable, and are shellable relative to the intersection with $\partial Y$.
\end{lem}

\begin{proof}
This is really a simplicial lemma, as links of faces in cubical complexes are simplicial complexes, and cubical stellar subdivisions on the cubical complex result in simplicial stellar subdivisions in the links: We begin by the considering the maximal face $\tau$ whose link is not shellable. We then perform stellar subdivisions at those faces containing it: newly introduced faces clearly have links that are shellable, so it remains to see what happens in the link of $\tau$.

It therefore suffices to prove that a PL sphere or ball becomes shellable after iterated derived subdivision that does not affect the boundary. For spheres, this is a result of \cite{AB-SSZ}. For the case of a ball $B$, we observe that stellar subdivision in the interior face adjacent to a boundary face (in the sense that $F$ is an interior face and a facet of $F$ lies in the boundary), followed by a shelling step at the boundary, is isomorphic to a subdivision step at the boundary. Hence, restricting to subdivisions at interior faces does not pose a restriction.
\end{proof}

Next, we observe that the cobordism may be chosen so that $M \times [0,1]$ collapses to $M \times \{0\}$. Indeed, we have the following result following from the main result of \cite{1107.57891}:

\begin{lem} After some iterated cubical stellar subdivision of interior faces,  $Y$ collapses to $Y \cap (M \times \{0\})$.
\end{lem} 

\begin{proof} Following \cite{1107.57891}, there exists a simplicial stellar subdivision $Y'$ of the interior faces of $Y$ so that $Y'$ collapses to $Y' \cap (M \times \{0\})$.

Now, to every ordinary stellar subdivision, we have an associated a cubical stellar subdivision.
 
Recall that this \emph{cubical} stellar subdivision $Y''$ of the interior faces of $M \times [0,1]$ admits a polyhedral map to $Y'$, that is, a map that sends polyhedra to polyhedra: This is the confiment map of Section~\ref{ssc:confinement}. We claim this map preserves collapses:

Recall that a collapse consists of a pair of faces $(A,B)$, such that $B$ is the only face that strictly contains $A$ (and $A$ is called free in this case)

Now, if $A$ is of positive dimension, then the preimage of both $A$ and $B$ are unique faces by Lemma~\ref{lem:confinement}, and the preimage of $A$ is hence free as well.

If $A$ is of dimension $0$, then it's preimage $\varphi^{-1} A$ is of the form $\mr{St}_\tau C$ that collapses to the intersection with the closure of $\varphi{-1} B$ by Lemma~\ref{lem:starcollapse}.

Hence $Y''$ collapses to $Y'' \cap (M \times \{0\})$.
\end{proof}

Hence, we have a cubulation $Y''$ of $M \times [0,1]$ which collapses to $Y'' \cap (M \times \{0\})$ and such that $Y'' \cap (M \times \{0\}) = Y'' \cap (M \times \{0\}) = X_0$, and in which we can assume that the star of any face is shellable, and shellable relative to the boundary as well.

Hence, using Proposition~\ref{prp:Whiteh}, after sufficiently many cubical derived subdivisions, we have a cubulation of $M \times [0,1]$ which shells to $M \times \{0\}$ and $M \times \{0\} = X_0$, $M \times \{1\} = X_1$. This gives the desired implication.
\qed

{\small
\bibliographystyle{myamsalpha}
\bibliography{Ref}}

\end{document}